\documentclass[11pt]{article}

\usepackage{amsmath,amsthm,amssymb,amsxtra,amsfonts,amsbsy,mathrsfs}
\usepackage[all,2cell]{xy} \UseAllTwocells
\SilentMatrices
\newtheorem{thm}[subsection]{Theorem}
\newtheorem{cor}[subsection]{Corollary}
\newtheorem{lemma}[subsection]{Lemma}
\newtheorem{prop}[subsection]{Proposition}
\newtheorem{sublemma}[subsubsection]{Lemma}

\newtheorem{subcor}[subsubsection]{Corollary}
\newtheorem{subthm}[subsubsection]{Theorem}
\theoremstyle{definition}
\newtheorem{notation-convention}[subsection]{Notations and
Conventions}
\newtheorem{defn}[subsection]{Definition}

\newtheorem{remark}[subsection]{Remark}

\newtheorem{anitem}[subsubsection]{}
\newtheorem{blank}[subsection]{}

\numberwithin{equation}{subsection}

\newcommand{\leftexp}[2]{{\vphantom{#2}}^{#1}{#2}}

\newcommand{\s}{\mathscr}
\newcommand{\bb}{\mathbb}

\newcommand{\ol}{\overline}
\newcommand{\Qlb}{\overline{\mathbb Q}_{\ell}}
\newcommand{\Z}{\bb Z}
\newcommand{\Q}{\bb Q}
\newcommand{\C}{\bb C}

\newcommand{\Gal}{\mathrm{Gal}}
\newcommand{\Hom}{\mathrm{Hom}}
\newcommand{\Spec}{\mathrm{Spec}}

\newcommand{\IC}{\mathrm{IC}}

\newcommand{\Fr}{\mathrm{Frob}}

\newcommand{\Tr}{\mathrm{Tr}}
\newcommand{\red}{\mathrm{red}}

\newcommand{\Perv}{\text{Perv}}

\begin{document}

\title{Independence of $\ell$ for the supports in
the Decomposition Theorem}
\author{Shenghao Sun\thanks{Yau Mathematical Sciences Center, Tsinghua University, Beijing 100084, China;
email: \texttt{shsun@math.tsinghua.edu.cn}. Partially supported by China NSF grant (11531007).}}
\date{}
\maketitle

\begin{abstract}
In this note, we prove a result on the independence of $\ell$
for the supports of irreducible perverse sheaves occurring in
the Decomposition Theorem, as well as for the family of local
systems on each support.
It generalizes Gabber's result on the independence
of $\ell$ of intersection cohomology to the relative case.
\end{abstract}

\section{Introduction}

We work with \'etale cohomology theory (or, more accurately, $\ell$-adic cohomology) of schemes here.
Let $f:X\to Y$ be a proper morphism of algebraic varieties over a
separably closed field $k$, and let $\IC_X(\Qlb)$ be the $\ell$-adic intersection complex on $X$. By the decomposition theorem \cite[Th\'eor\`eme 6.2.5]{BBD}, $Rf_*\IC_X(\Qlb)$ splits into a
direct sum of shifted semisimple perverse sheaves. (Although the theorem cited above is stated for $k=\C$, the proof works in general.) It is then
natural to expect that the supports of the irreducible perverse
sheaves occurring in the decomposition are independent of
$\ell\ne p$. This is clear if $\text{char }k=0$, by
reducing to the case where $k=\C$ and comparing with $\C$-coefficients.

The question can be raised from another perspective, and made more general. One has the following conjecture on ``independence of $\ell$" under pushforwards. Given an $\bb F_q$-morphism $f:X\to Y$ between $\bb F_q$-schemes of finite type, and given a compatible system $\{\s F_i\}_I$ of irreducible lisse sheaves on $X$ (see Definition \ref{D1} (i) below for definition), one predicts that for each integer $n$, the family of constructible sheaves $\{R^nf_*\s F_i\}_I$ are lisse on the strata of a common stratification of $Y$ and that on each stratum, their restrictions are compatible; similarly for the other operations. Note that the assumption on lissity of $\s F_i$ is necessary, since a lisse sheaf (which may have no global sections) could be compatible with the direct sum of a punctual sheaf (which always have global sections) and another constructible sheaf; similarly, the assumption on irreducibility rules out for instance nontrivial extensions of constant sheaves. Some results in this direction are given by Illusie \cite{Ill}, who showed, in particular, that there is a stratification $\{Y_i\}_i$ of $Y$, such that all the sheaves $R^nf_*\Z_{\ell}$ (as well as $R^nf_!\Z_{\ell}$), as $n$ and $\ell\ne p$ vary, are lisse on each $Y_i$ (see \textit{loc.\ cit.}, Corollaire 2.7).

Our main result Theorem \ref{T1} could be regarded as a variant of this conjecture with respect to the perverse $t$-structure; namely, we prove that if $f$ is proper, then for each $n$, the family of semisimplified perverse sheaves $\leftexp{p}{R^nf_*}(\s F_i)^{\text{ss}}$ is ``independent of $i$": the supports of their irreducible constituents are the same, and the family of local systems on each support is compatible. As a corollary, we have

\begin{thm}\label{T-main}
For a proper morphism $f:X\to Y$ of $\ol{\bb F}_p$-schemes of finite type, the set of supports occurring in the
decomposition of $Rf_*\IC_X(\Qlb)$ is
independent of $\ell$. Moreover, as $\ell\ne p$ varies,
the local systems occurring in the decomposition over each stratum
form a compatible system \emph{(}with respect to some model defined over a finite subfield of $\ol{\bb F}_p$\emph{)}.
\end{thm}

This can be regarded as the
relative version of a result of Gabber \cite{Fuj}.

\begin{blank}\label{notation}
We fix some notations and conventions.

\begin{anitem}\label{cd_l}
We work over a finite field $\bb F_q$ of characteristic $p$. Here $``\ell"$, as well as $``\ell_i"$ in the sequel, denote prime numbers different from $p$. Let $\bb F$ be a fixed algebraic closure of $\bb F_q$.
We always assume $k$-schemes to be of finite type over their base field $k$. We only work with \textit{middle} perversity \cite[Section 4]{BBD}.
\end{anitem}

\begin{anitem}\label{ess-smooth}
Let $X$ be a $k$-scheme. We say that $X$ is
\textit{essentially smooth} if $(X_{\ol{k}})_{\red}$ is
smooth over $\ol{k}$. When $k$ is perfect (e.g.\ a finite field), this is equivalent to $X_{\red}$ being regular. A $\Qlb$-lisse sheaf is sometimes called an \textit{$\ell$-adic local system}.
\end{anitem}

\begin{anitem}\label{IC}
Let $X$ be a $k$-scheme, and let $Z$ be an
essentially smooth irreducible locally closed subscheme of
$X$. Let $j:Z\to\ol{Z}$ be the open immersion of $Z$ into
its closure in $X$. Let $L$ be an $\ell$-adic lisse sheaf on $Z$.
We will denote the intermediate extension $j_{!*}L[\dim Z]$
(see \cite[p.\ 58]{BBD}), or its extension-by-zero to $X$,
by $\IC_{\ol{Z}}(L)$, and call it the \textit{intersection
complex defined by the local system} $L$.

We also denote the closure of $Z$ in $X$ by $Cl_X(Z)$, when there is a possible confusion (e.g.\ in Lemma \ref{L0-middle-ext}).
\end{anitem}

\begin{anitem}\label{max-support}
Let $X$ be an irreducible $k$-scheme, $j:U\to X$
an essentially smooth open subset of $X$, and let $L$ be
an $\ell$-adic local system on $U$. We say that $U$ is the
\textit{maximal support} of $L$, if for any essentially
smooth open subset $V$ of $X$, such that the local system
$L|_{U\cap V}$ can be extended to a local system on $V$, we
have $V\subset U$. Given $L$ on $U$, its maximal support
always exists: it is the union of all essentially smooth open sets of $X$ over which the sheaf $j_*L$ is lisse.
\end{anitem}

\begin{anitem}\label{closed-supp}
Let $X$ be a $k$-scheme and let $P$ be an $\ell$-adic perverse sheaf on $X$, with irreducible constituents
$\IC_{\ol{X}_{\alpha}}(L_{\alpha})$, where for each $\alpha,\ L_{\alpha}$ is an irreducible local system with maximal support $X_{\alpha}\subset\ol{X}_{\alpha}$. Then we say that
$X_{\alpha}$ (resp.\ $\ol{X}_{\alpha}$) is a \textit{support}
(resp.\ a \textit{closed support}) \textit{occurring in} $P$.
\end{anitem}

\begin{anitem}\label{K-gp}
The Grothendieck group $K(X,\Qlb)$ of $D_c^b(X,\Qlb)$, which is also the Grothendieck group of $\text{Perv}(X,\Qlb)$, is isomorphic to the free abelian group generated by the isomorphism classes of irreducible perverse sheaves on $X$ (see \cite[0.8]{Lau}). This is the case for any noetherian and artinian abelian category.

For $K\in D^b_c(X,\Qlb)$, let $[K]$ be its image in the Grothendieck group. Then we have
\[
[K]=\sum_n(-1)^n\,[\leftexp{p}{\s H^nK}].
\]
For the given $K\in D^b_c(X,\Qlb)$, there exist two
semisimple perverse sheaves $P^+$ and $P^-$ without common
irreducible factors, such that
\[
[K]=[P^+]-[P^-],
\]
and they are unique up to isomorphism. We shall call it the
\textit{canonical representative} of $[K]$. It is obtained from the identity above after cancellation.
\end{anitem}

\begin{anitem}\label{Weil-cplx}
When $k$ is a finite field, we work with \textit{Weil sheaves} \cite[D\'efinition 1.1.10]{Del2} and \textit{Weil complexes} \cite[2.4.2]{Sun1} in this article. The category of bounded $\ell$-adic Weil complexes on $X$ over $\bb F_q$ is denoted $W^b(X,\Qlb)$. Although this is not a triangulated category,
one can still talk about exact triangles in it: they are diagrams
\[
\xymatrix@C=.5cm{
(K_1,\varphi_1) \ar[r] & (K_2,\varphi_2) \ar[r] & (K_3,\varphi_3) \ar[r] & (K_1[1],\varphi_1[1])}
\]
such that $K_1\to K_2\to K_3\to K_1[1]$ is an exact triangle in $D^b_c(X_{\bb F},\Qlb)$.
Perverse sheaves are understood in this setting, and the notions in (\ref{max-support},\ \ref{closed-supp},\ \ref{K-gp}) above generalize to perverse Weil sheaves $\Perv^W(X,\Qlb)$. For instance, the Grothendieck group $K^W(X,\Qlb)$ is the free abelian group generated by isomorphism classes of irreducible perverse Weil sheaves.
Let $W^b_m(X,\Qlb)$ be the full subcategory of $W^b(X,\Qlb)$ consisting of mixed Weil complexes.
\end{anitem}

\begin{anitem}\label{char_poly}
Let $X$ be an $\bb F_q$-scheme, and $x\in X(\bb F_{q^v})$ (or $x\in |X|$, namely a closed point). Let $\s F$ be a $\Qlb$-sheaf on $X$. We denote by $P_x^{\s F}(t)$ the polynomial \cite[1.1.8]{Del2}
\[
\det(\textbf{1}-\Fr_x\cdot t,\s F)\in\Qlb[t],
\]
and define the \textit{$L$-function} $L(\s F,t)$ to be
\[
L(\s F,t)=\prod_{x\in|X|}\ \frac{1}{P_x^{\s F}(t^{\deg(x)})},
\]
where $\deg(x)=[k(x):\bb F_q]$ is the degree of $x$.
Both extend to $\s F\in W^b(X,\Qlb)$:
\[
P_x^{\s F}(t):=\prod_{i\in\Z}\ P_x^{\s H^i(\s F)}(t)^{(-1)^i}\in\Qlb(t),
\]
and depends only on the image $[\s F]$ of $\s F$ in the Grothendieck group of $\ell$-adic Weil sheaves.
\end{anitem}

\begin{anitem}\label{Tate-twist}
Given $b\in\Qlb^\times$, denote by $\Qlb^{(b)}$ the Weil sheaf on $\Spec\ \bb F_q$ of rank 1 on which the geometric Frobenius acts as multiplication by $b$. Given an $\bb F_q$-scheme $X$ with structural map $a:X\to\Spec\ \bb F_q$, and an $\ell$-adic Weil sheaf $\s F$ on $X$, let $\s F^{(b)}$ be $\s F\otimes a^*\Qlb^{(b)}$, called a \textit{Tate twist deduced from $\s F$}. (See \cite[1.2.7]{Del2}, but be aware of the difference that, to simplify the notation, we work over a fixed finite field $\bb F_q$, rather than over $\bb F_p$.)
\end{anitem}

\begin{anitem}\label{anitem-stack}
For the perverse $t$-structure on $k$-Artin stacks (always assumed of finite presentation), see \cite{LO3}; for the $L$-functions of $\bb F_q$-Artin stacks, see \cite[Definition 4.1]{Sun1}. The other notions above then all generalize without much difficulty to Artin stacks.
\end{anitem}
\end{blank}

\textbf{Acknowledgement.} I would like to thank Ofer Gabber, Martin Olsson and Weizhe Zheng for helpful discussions. In particular, Weizhe Zheng showed me how to generalize the result to algebraic stacks. The referees pointed out some mistakes in early versions, and gave valuable comments helping to improve the readability of the article greatly; in particular, the necessity of proving a \v{C}ebotarev density theorem for Weil sheaves (Lemma \ref{L-Chebotarev}) was pointed out by one of the referees. I appreciate them all.
Part of this work was done during the stay in
Universit\'e Paris-Sud (UMR 8628), supported by ANR grant
G-FIB, and in IH\'ES. This work is also partially supported by China NSF grant (11531007).

\section{Perverse compatible systems}

In this section, we review the notion of a compatible system and define its variant for the perverse $t$-structure, called a \emph{a perverse compatible system}, and show their relations.

Let $E$ be a number field, and let $I$ be a set of pairs
$i=(\ell_i,\sigma_i),$ where $\ell_i$ is a rational prime
number not equal to $p$, and $\sigma_i:E\hookrightarrow\ol{\bb
Q}_{\ell_i}$ is an embedding of fields.
Note that, to give an embedding $E\to\Qlb$ is the same as giving a finite place $\lambda$ of $E$ over $\ell$.

Part (i) in the following definition, at least for $K_i$ sheaves,
is well-known; see for instance
\cite[Chapter I, Section 2.3]{Serre} for number fields.

\begin{defn}\label{D1}
Let $X$ be an $\bb F_q$-scheme, and for
each $i\in I,$ let $K_i\in W^b(X,\ol{\bb Q}_{\ell_i})$.

(i) We say that $\{K_i\}_I$ is a \emph{weakly $(E,I)$-compatible system}, if for every integer $v\ge1$ and for every point $x\in X(\bb F_{q^v})$, there exists a number $t_x\in E$ such that
\[
\sigma_i(t_x)=\Tr(\Fr_x,K_i)
\]
for all $i\in I$.

(ii) We say that $\{K_i\}_I$ is a \emph{strongly $(E,I)$-compatible system}, if for each $n\in\bb Z,$ the system of cohomology sheaves $\{\s H^nK_i\}_I$ is a weakly $(E,I)$-compatible system.

(iii) Assume that $K_i=P_i$ are perverse sheaves. We say that
$\{P_i\}_I$ is \emph{perverse $(E,I)$-compatible}, if there exist a finite number of essentially smooth irreducible locally closed subschemes $X_{\alpha}\hookrightarrow X$, and for each $\alpha$ a weakly $(E,I)$-compatible system $\{L^i_{\alpha}\}_I$ of semisimple local systems on $X_{\alpha}$, such that each irreducible factor of $L_\alpha^i$ has $X_{\alpha}$ as its maximal support (inside $\ol{X}_{\alpha}$), and that
\[
P_i^{\text{ss}}\simeq\bigoplus_{\alpha}\ \IC_{\ol{X}_{\alpha}} (L^i_{\alpha})
\]
for all $i$. Here $P_i^{\text{ss}}$ denotes the semi-simplification of $P_i$ in the abelian category of perverse sheaves.

(iv) Once again, let $K_i\in W^b(X,\ol{\bb Q}_{\ell_i})$. We say that $\{K_i\}_I$ is a \emph{weakly perverse $(E,I)$-compatible system} if, letting
\[
[K_i]=[P^+_i]-[P^-_i]
\]
be the canonical representatives (\ref{K-gp}), both $\{P^+_i\}_I$ and
$\{P^-_i\}_I$ are perverse $(E,I)$-compatible.

(v) We say that $\{K_i\}_I$ is a \emph{strongly perverse $(E,I)$-compatible system,} if for each $n\in\bb Z$, the system of perverse cohomology sheaves $\{\leftexp{p}{\s H^n(K_i)}\}_I$ is perverse $(E,I)$-compatible.
\end{defn}

These notions apply to $\bb F_q$-Artin stacks $X$ as well.

\begin{remark}\label{R1}
(i) Basically a ``weak" version of the notion depends only on the
image in the Grothendieck group, so there could be cancellations
between cohomologies of even indices and odd indices, and the weak notion could not see shifts of indices by even integers. A ``strong" version is for the cohomology objects, so cancellations and shifts are not allowed.

(ii) For a system $\{\s F_i\}_I$ of sheaves, the notions (i) and
(ii) in Definition \ref{D1} are the same, and we simply say that $\{\s F_i\}_I$ is an \emph{$(E,I)$-compatible system of sheaves.}

(iii) The condition in Definition \ref{D1} (i) is equivalent to the one given in \cite[1.2]{Fuj}: for each closed point $x\in|X|$, there exists a rational function $P_x(t)\in E(t)$ such that
\[
\sigma_i(P_x(t))=P_x^{K_i}(t),\ \forall i\in I.
\]
This is because formally we have
\[
-t\frac{d}{dt}\log P_x^K(t)=\sum_{r=1}^{\infty} \Tr(\Fr_x^r,K)t^r.
\]
The advantage of our definition is that it applies to Artin stacks as well.

(iv) Note that, in the situation of (\ref{max-support}),
if $L=L_1\oplus\cdots\oplus L_n$ is a direct sum of sub-local systems, then the maximal support of $L$ is the intersection of those of the $L_i$'s; in particular, the maximal support of a direct factor of $L$ could be larger than that of $L$. In Definition \ref{D1} (iii) above, the condition we impose is then different from asking the $L_\alpha^i$'s to have maximal support $X_\alpha$.
\end{remark}

\begin{lemma}\label{L-poids}
Let $X$ be a connected $\bb F_q$-scheme and let $\{\s F_i\}_I$ be an $(E,I)$-compatible system of punctually pure lisse sheaves on $X$; assume $\s F_i\ne0$. Let $w_i\in\Z$ be the weight of $\s F_i$.
Then the numbers $w_i\ (i\in I)$ are all equal.
\end{lemma}

\begin{proof}
Take any closed point $x\in|X|$; then by Remark \ref{R1} (iii) there is a polynomial $P_x(t)\in E[t]$ such that
\[
\sigma_i(P_x(t))=P_x^{\s F_i}(t),\ \forall i\in I.
\]
The weight of $\s F_i$ is then the weight of any reciprocal root of $P_x(t)$ relative to $q^{\deg(x)}$, with respect to an arbitrary embedding $E\to\C$, hence independent of $i$.
\end{proof}

Following \cite[0.9]{FrobTr}, we give the following definition.

\begin{defn}\label{D2-alg}
For an $\bb F_q$-scheme $X$ and $K\in W^b(X,\Qlb)$, we say that $K$ is \textit{algebraic}, if $\Tr(\Fr_x,K)$ is an algebraic number in $\Qlb$ (i.e.\ algebraic over $\Q$), for any $v\ge1$ and $x\in X(\bb F_{q^v})$.
\end{defn}

Clearly, being algebraic or not depends only on the image $[K]$ of $K$ in the Grothendieck group of Weil sheaves.
If $\{K_i\in W^b(X,\ol{\Q}_{\ell_i})\}_I$ is a weakly $(E,I)$-compatible system on $X$, then each $K_i$ is algebraic (since $E$ consists of algebraic numbers). Mixed complexes are by definition algebraic.

Now we recall the following theorem of Drinfeld, based upon earlier results of Chin and Deligne, which is the key ingredient of this article.

\begin{thm}\label{Drinfeld}
Let $X$ be a smooth connected $\bb F_q$-scheme, and let $\s F$ be an irreducible lisse $\ell$-adic sheaf on $X$, whose determinant has finite order. Then there exists a number field $E$ with an embedding $\sigma:E\hookrightarrow\Qlb$, such that for each point $x\in|X|$, the polynomial $P_x^{\s F}(t)\in\Qlb[t]$
has coefficients in $\sigma(E)$, and that for every finite place
$\lambda$ of $E$ not lying over $p$, there exists a lisse $E_{\lambda}$-sheaf $\s G$ which is compatible with $\s F$.
\end{thm}

This follows from \cite[Th\'eor\`eme 1.6]{FrobTr}, \cite[Theorem 1.1]{Dri} and \cite[Main Theorem]{Chin02}. Note that such a $\s G$ is necessarily unique and irreducible:
it is irreducible by looking at the order of the
$L$-function $L(\s G\otimes\s G^{\vee},t)$ at $t=q^{-\dim X}$
(the irreducibility is also mentioned in the statement of \textit{loc.\ cit.}).
If we know that such a lisse sheaf $\s G$ must be an \'etale sheaf (rather than just a Weil sheaf),
then it is determined, according to the usual \v{C}ebotarev's density, by its local characteristic polynomials
$P_x^{\s G}(t)$. As the polynomial $P_x^{\det(\s G)}(t)$ is determined by $P_x^{\s G}(t)$, we see that
$\det(\s G)$ is compatible with $\det(\s F)$, hence it is also of finite order, therefore $\det(\s G)$, as well as $\s G$ itself
by \cite[Proposition 1.3.14]{Del2}, are \'etale sheaves.
It is sometimes called an \textit{$E_{\lambda}$-companion} of $\s F$.

In the following, for the uniqueness of companions, we need a ``\v{C}ebotarev's density theorem" for semisimple lisse Weil sheaves
on normal $\bb F_q$-schemes. Note that, in contrast to the fundamental group, the local Frobenius elements are not dense
in the Weil group.
One could prove a more general version for the Grothendieck group $K^W(X,\Qlb)$, analogous to the
formulation in \cite[$\S1$]{Fuj} or \cite[Th\'eor\`eme 1.1.2]{Lau}, but for our purpose, we only need this elementary version.

\begin{sublemma}\label{L-Chebotarev}
Let $X$ be an integral normal $\bb F_q$-scheme, and let $\s F_1$ and $\s F_2$ be two semisimple lisse $\ell$-adic sheaves
on $X$. If $P_x^{\s F_1}(t)=P_x^{\s F_2}(t)$ for all $x\in|X|$, then $\s F_1\simeq\s F_2$.
\end{sublemma}

Following \ref{Weil-cplx}, by sheaves we mean Weil sheaves here.

\begin{proof}
Let $\bb F_{q^r}$ be the algebraic closure of $\bb F_q$ in the function field of $X$. Because $X$ is normal, the structural morphism $X\to\Spec\ \bb F_q$ factors through $\Spec\ \bb F_{q^r}$. Replacing $\bb F_q$ by $\bb F_{q^r}$, we may assume that $X$ is geometrically integral. Then $X$ satisfies the hypothesis at the beginning of \cite[Section 1.3]{Del2} (namely, normal and geometrically connected).

Let $\s F_1\simeq L_1\oplus\cdots\oplus L_n$ be the decomposition into irreducible lisse sheaves. By \cite[1.3.6]{Del2}, for each $i=1,\cdots,n$, certain Tate twist $L_i^{(1/b_i)}$ of $L_i$, for some $b_i\in\Qlb^\times$, has determinant of finite order, so it is an \'etale sheaf \cite[Proposition 1.3.14]{Del2}.
Clearly, $L_i^{(u/b_i)}$ is also an \'etale sheaf for any $\ell$-adic unit $u$. If $x\in|X|$, and $\alpha_i$ is an eigenvalue of $\Fr_x$ on $L_i$, then $\alpha_i/b_i^{\deg(x)}$ is an eigenvalue on $L_i^{(1/b_i)}$, which, by \cite[Th\'eor\`eme 1.6]{FrobTr}, is a Weil number of weight 0 and, a fortiori, is an $\ell$-adic unit.
Let $v:\Qlb^\times\to\bb Q$ be the $\ell$-adic valuation on $\Qlb$, normalized such that $v(\ell)=1$. Then $v(\alpha_i)=\deg(x)v(b_i)$; in particular, $v(\alpha_i)$ is independent of the choice of the eigenvalue $\alpha_i$. Let $c_i=\ell^{v(\alpha_i)/\deg(x)}$. Here we have chosen, once and for all, a compatible family of roots $\ell^{1/k}$ of $\ell$ in $\Qlb$ (i.e.\ the $k'$-th power of the $(kk')$-th root is the $k$-th root). Then $b_i/c_i$ is an $\ell$-adic unit, so that $L_i^{(1/c_i)}$ is also an \'etale sheaf. Note that the value of $v(\alpha_i)/\deg(x)$ is also independent of the choice of $x$, and by \cite[Proposition 1.3.14]{Del2}, the number $c_i$,
up to $\ell$-adic units, is the only one such that $L_i^{(1/c_i)}$ is an \'etale sheaf.

Now we group the direct factors $L_i$ according to the values $v(\alpha_i)/\deg(x)$, to obtain
\[
\s F_1\simeq M_1\oplus\cdots\oplus M_r,
\]
where each $M_k$ is the direct sum of those $L_i$'s with the same value of $v(\alpha_i)/\deg(x)$. The polynomials $P_x^{M_k}(t)$ can be read off from $P_x^{\s F_1}(t)$, by looking at the $\ell$-adic valuations of the reciprocal roots of $P_x^{\s F_1}(t)$. Moreover, the Tate twist $M_k^{(1/c_{i_k})}$, where $L_{i_k}$ is a direct factor of $M_k$, is an \'etale sheaf, hence is determined, up to isomorphism, by the knowledge of all the polynomials $P_x^{M_k^{(1/c_{i_k})}}(t),\ x\in|X|$, by the usual \v{C}ebotarev's density theorem. Twisting back, we see that $\s F_1$ is determined by $P_x^{\s F_1}(t),\ x\in|X|$.
\end{proof}

\begin{subcor}\label{C-gal-twist}
Let $X$ be a smooth connected $\bb F_q$-scheme, and let $\s F$ be an algebraic lisse $\Qlb$-sheaf on $X$. Then there exists a number field $E$ with an embedding $\sigma:E\hookrightarrow\Qlb$, containing all the local Frobenius traces of $\s F$, and that for every embedding $\sigma':E\hookrightarrow\ol{\Q}_{\ell'},\ \ell'\ne p$, there exists a unique semisimple lisse $\ol{\Q}_{\ell'}$-sheaf, denoted $\s F^{\sigma'}$, which is compatible with $\s F$.
\end{subcor}

In the sequel, we also call $\s F^{\sigma'}$ the \textit{$\sigma'$-companion} of $\s F$.

\begin{proof}
The uniqueness follows from Lemma \ref{L-Chebotarev}. To see its existence, after semisimplification, we may assume that $\s F$ is semisimple, and let
\[
\s F=L_1\oplus\cdots\oplus L_n
\]
be the decomposition into irreducible lisse sheaves. Then each $L_i$ is algebraic: for each $x\in|X|$, the polynomial $P_x^{L_i}(t)$ divides $P_x^{\s F}(t)$, whose roots are algebraic numbers. This reduces us to the case where $\s F$ is irreducible: one simply takes $E$ to be the composite of all those fields $E_i\subset\Qlb$ corresponding to $L_i$, and takes $\s F^{\sigma'}$ to be $\oplus L_i^{\sigma'}$.

So we assume that $\s F$ is irreducible. Replacing $\bb F_q$ by its algebraic closure in the function field of $X$, we may assume that $X$ is geometrically irreducible. By \cite[1.3.6]{Del2}, certain Tate twist $\s F^{(1/b)}$ of $\s F$, for some $b\in\Qlb^\times$, has determinant of finite order, hence is algebraic, by \cite[Th\'eor\`eme 1.6]{FrobTr}. Therefore, the number $b$ is algebraic. We take $E\overset{\sigma}{\hookrightarrow}\Qlb$ to be the subfield generated by the number $b$ and a subfield for $\s F^{(1/b)}$, given by Theorem \ref{Drinfeld}; it clearly contains all the local Frobenius traces of $\s F$. For any embedding $\sigma':E\hookrightarrow\ol{\Q}_{\ell'}$, with $\s G$ the $\sigma'$-companion of $\s F^{(1/b)}$, again given by Theorem \ref{Drinfeld}, we take $\s F^{\sigma'}$ to be the Tate twist $\s G^{(\sigma'(b))}$ of $\s G$.
\end{proof}

One observes that, if $\s F$ is mixed and lisse, then $\s F^{\sigma'}$ is also mixed, and they have the same weights.

The following lemma shows that, compatible semisimple local systems have the same maximal support.
An \textit{extension} of the pair $(E,I)$ is another pair $(E',I')$ together with a bijection $I\to I'$ of the form
\[
(\ell_i,\sigma_i)\mapsto(\ell_i,\sigma_i'),
\]
where $E'/E$ is a finite extension and for each $i\in I,\ \sigma_i'$ is an extension of $\sigma_i$ to $E'$.
In other words, if we interpret the embeddings $\sigma_i$ (resp.\ $\sigma_i'$) as finite places $\lambda_i$ of $E$ (resp.\ finite places $\lambda_i'$ of $E'$), then we ask $\lambda_i'$ to be a place of $E'$ lying over $\lambda_i$.

\begin{lemma}\label{L0-supp}
Let $X$ be an irreducible $\bb F_q$-scheme, and let $U_1$ and $U_2$ be two essentially smooth open subsets in $X$. For $i=1,2$, let $L_i$ be a semisimple lisse $\ol{\Q}_{\ell_i}$-sheaf on $U_i$, having $U_i$ as maximal support, and let $\sigma_i:E\to\ol{\bb Q}_{\ell_i}$ be an embedding. Let $I=\{(\ell_i,\sigma_i)|i=1,2\}$. If $L_1$ and $L_2$ are $(E,I)$-compatible on some non-empty open set $U\subset U_1\cap U_2$, then $U_1=U_2$.
\end{lemma}

\begin{proof}
After making an extension of $(E,I)$ if necessary, we may assume, by Corollary \ref{C-gal-twist}, that $L_1$ has a $\sigma_2$-companion $M$ on $U_1$, and vice versa for $L_2$.
Then $M|_U$ and $L_2|_U$ are semisimple, so they are both
$\sigma_2$-companions of $L_1|_U$,
and by uniqueness we have $M|_U\simeq L_2|_U$. As $\pi_1(U)$ is mapped onto both $\pi_1(U_1)$ and $\pi_1(U_2)$, we see that $j_*L_2$ is lisse over $U_1\cup U_2$, where $j:U_2\to X$ is the open immersion. So $U_1\subset U_2$, by the maximality of $U_2$. By symmetry we have $U_2\subset U_1$ as well.
\end{proof}

Now we prove two key propositions, relating compatibilities in the natural $t$-structure and the perverse $t$-structure. For Proposition \ref{L1}, the proof, in a large part, focuses on the case when $X$ is integral. To reduce the general case to the integral case, we need Lemma \ref{L0-middle-ext}, whose proof repeats part of the integral case. So we put Lemma \ref{L0-middle-ext} after Proposition \ref{L1}, to reduce the repetition.

\begin{prop}\label{L1}
Let $\{P_i\}_I$ be a system of mixed perverse sheaves on an $\bb F_q$-scheme $X$. If they are weakly $(E,I)$-compatible, then they are perverse $(E,I)$-compatible.
\end{prop}

\begin{proof}
It suffices to prove it for $I=\{1,2\}$.
Since $[P_i]=[P_i^{\text{ss}}]$, we may assume that the $P_i$'s are semisimple perverse sheaves. We may replace $X$ by the union of the closed supports occurring in $P_1$ or $P_2$, with the reduced subscheme structure. Let
\[
X=\bigcup_{i=1}^rX_i
\]
be the decomposition of $X$ into irreducible components, and put $X_1^0=X-\cup_{i=2}^rX_i$. Let $j:X_1^0\to X$ be the open immersion. Then $\{j^*P_i\}_I$ is a weakly $(E,I)$-compatible system of semisimple mixed perverse sheaves on $X_1^0$. If we can show that $\{j^*P_i\}_I$ is perverse $(E,I)$-compatible, then by Lemma \ref{L0-middle-ext} (i), to be proved in the following, the system $\{j_{!*}j^*P_i\}_I$ is perverse $(E,I)$-compatible, as well as weakly $(E,I)$-compatible by Lemma \ref{L0-middle-ext} (ii). Then we use induction on the length of $P_i$: by induction hypothesis, $\{P_i/j_{!*}j^*P_i\}_I$, being weakly $(E,I)$-compatible, is perverse $(E,I)$-compatible, and therefore
\[
P_i\simeq(j_{!*}j^*P_i)\oplus(P_i/j_{!*}j^*P_i),\quad i=1,2,
\]
is also perverse $(E,I)$-compatible. This allows us to replace $X$ by $X_1^0$, assuming $X$ to be irreducible.

Either $P_1$ or $P_2$ has an open support (since $X$ is the union of the closed supports), and so does the other, as one sees by restricting them to an open subset.
Let $X^0$ be the intersection of all supports occurring in either $P_1$ or $P_2$ that are open, and let $Y$ be the union of all proper (i.e.\ $\neq X$) closed supports occurring in $P_1$ or $P_2$. We put $U=X^0-Y$, and let $\s F_i$ be the semisimple mixed
$\ol{\Q}_{\ell_i}$-local system
\[
\s F_i=\s H^{-\dim X}(P_i|_U)=(P_i|_U)[-\dim X]
\]
on $U$. Then $U$ is regular and irreducible, and
$\s F_1$ and $\s F_2$ are $(E,I)$-compatible. The point is to show that,
after grouping the irreducible factors of $\s F_i$ according to their maximal supports
inside $X$, the groups, one by one, are still $(E,I)$-compatible.

Let
\[
\s F_1=L_1\oplus\cdots\oplus L_n\qquad(\text{resp.\ }\s F_2=M_1\oplus\cdots\oplus
M_m)
\]
be the decomposition into irreducible lisse sheaves. By \cite[Th\'eor\`eme 3.4.1]{Del2}, these direct factors are punctually pure.
And by the $(E,I)$-compatibility, the weights of $\s F_1$ (with multiplicities) coincide with the weights of $\s F_2$. Let $L_1$ be of maximal weight among the $L_j$'s.

By Corollary \ref{C-gal-twist}, there is a finite extension $E'$ of $E$, which we may take to be Galois, embedded in $\ol{\Q}_{\ell_1}$ via some (any) fixed extension $\sigma_1'$ of $\sigma_1$, containing all local traces of the $L_j$'s, such that for any embedding $\sigma':E'\to\ol{\Q}_{\ell'},\ \ell'\ne p$, and for any $1\le j\le n$, the $\sigma'$-companion $L_j^{\sigma'}$ of $L_j$ exists. Clearly, $L_j^{\sigma'}$ is punctually pure of the same weight as $L_j$.
Note that, given an embedding of $E$ into an algebraically closed field $F$, such as $\ol{\Q}_{\ell_i}$, extensions to $E',\ \Hom_E(E',F)$, form a torsor under $\Gal(E'/E)$. We fix an embedding $\sigma_2':E'\to\ol{\Q}_{\ell_2}$, extending $\sigma_2$, and let $I'=\{(\ell_1,\sigma_1'),(\ell_2,\sigma_2')\}$.

We claim that

$\bullet$ for any $\tau\in\Gal(E'/E)$, giving rise to another extension $\sigma_1'\circ\tau$ of $\sigma_1$, and for any $j$, the $(\sigma_1'\circ\tau)$-companion $L_j^{\sigma_1'\circ\tau}$ is isomorphic to $L_{j'}$ for some $j'$, and that

$\bullet$ the $\sigma_2'$-companion $L_1^{\sigma_2'}$ is isomorphic to some $M_j$.

The first one is clear: we have
\[
\s F_1^{\sigma_1'\circ\tau}\simeq L_1^{\sigma_1'\circ\tau} \oplus\cdots\oplus L_n^{\sigma_1'\circ\tau},
\]
in which each $L_j^{\sigma_1'\circ\tau}$ remains irreducible (using $L$-functions), and $\s F_1\simeq\s F_1^{\sigma_1'\circ\tau}$, because $\s F_1$ is semisimple with local Frobenius traces contained in $\sigma_1(E)$. The second claim follows from an argument using $L$-functions. We see that $\s F_2\otimes (L_1^{\sigma_2'})^\vee$ is $(E',I')$-compatible with
$\s F_1\otimes L_1^\vee$. Then we look at the order of the $L$-function $L(\s F_2\otimes(L_1^{\sigma_2'})^\vee,t)$ at $t=q^{-\dim U}$. By the Grothendieck-Lefschetz fixed-point formula, we have
\begin{equation*}
\begin{split}
L(\s F_2\otimes(L_1^{\sigma_2'})^\vee,t)
&=\prod_{j=1}^mL(M_j\otimes(L_1^{\sigma_2'})^\vee,t) \\
&=\prod_{j=1}^m\quad\prod_{i=0}^{2\dim U}\det(\textbf{1}-\Fr_q\cdot t,H^i_c(U_{\bb F},M_j\otimes(L_1^{\sigma_2'})^\vee))^{(-1)^{i-1}},
\end{split}
\end{equation*}
and the lisse sheaves $M_j\otimes(L_1^{\sigma_2'})^\vee$, which are all punctually pure, have only non-positive weights, by our choice of $L_1$. By Deligne's theorem on weights \cite[Corollaire 3.3.4]{Del2}, only those factors with $i=2\dim U$ could possibly contribute to the order at $t=q^{-\dim U}$, and we have
\begin{equation*}
\begin{split}
-\text{ord}_{t=q^{-\dim U}}L(\s F_2\otimes(L_1^{\sigma_2'})^\vee,t)
&=\sum_{j=1}^m \dim\big(H^0(U_{\bb F},M_j^\vee\otimes L_1^{\sigma_2'}(\dim U))^\vee\big)^{\Fr_q=q^{\dim U}} \\
&=\sum_{j=1}^m \dim\big(H^0(U_{\bb F},M_j^\vee\otimes L_1^{\sigma_2'})^\vee\big)^{\Fr_q=1} \\
&=\sum_{j=1}^m\dim H^0(U,M_j^\vee\otimes L_1^{\sigma_2'}) \\
&=\sum_{j=1}^m\dim \Hom_U(M_j,L_1^{\sigma_2'}) \\
&=\#\{1\le j\le m\ |\ M_j\simeq L_1^{\sigma_2'}\}.
\end{split}
\end{equation*}
Here ``$V^{\varphi=\lambda}$" denotes the eigenspace of an operator $\varphi$ on a vector space $V$ with eigenvalue $\lambda$, and for the first equality to hold we need the fact that the lisse sheaves $M_j\otimes(L_1^{\sigma_2'})^\vee$ are semisimple
(the tensor product of two semisimple local systems is again semisimple: see \cite[Chapitre IV, Proposition 5.2]{Chev}).

Say $M_1\simeq L_1^{\sigma_2'}$. Then $\oplus_{j=2}^nL_j$ is $(E',I')$-compatible with $\oplus_{j=2}^mM_j$. Continuing this process, we see that $n=m$, and that $L_j$ and $M_j$ are $(E',I')$-compatible after a possible relabelling.

Let $H\subset\Gal(E'/E)$ be the subgroup of all $\tau$'s such that $L_1^{\sigma_1'\circ\tau}\simeq L_1$; equivalently, it is the subgroup that, under the Galois correspondence, corresponds to the intermediate field of $\sigma_1'(E')/\sigma_1(E)$ generated by all local Frobenius traces $\Tr(\Fr_x,L_1),\ x\in U(\bb F_{q^v})$. Let
\[
\Gal(E'/E)=\coprod_{a=1}^h\quad\tau_aH
\]
be the right-$H$-cosets decomposition. We define
\[
\widetilde{L}_1:=\bigoplus_{a=1}^h\ L_1^{\sigma_1'\circ\tau_a},
\]
and say that the irreducible factors of $\widetilde{L}_1$ are \textit{Galois conjugate} to one another.
These direct factors are mutually non-isomorphic. It is independent, up to isomorphism, of the choice of the coset representatives $\tau_a$. Its local Frobenius traces, a priori in $\sigma_1'(E')$, are invariant under $\Gal(E'/E)$, hence are contained in $\sigma_1(E)$. In fact, for any $\tau\in\Gal(E'/E)$,
also viewed as an automorphism of $\sigma_1'(E')/\sigma_1(E)$, and for any $x\in U(\bb F_{q^v})$, we have
\[
\tau\Tr(\Fr_x,L_1^{\sigma_1'\circ\tau_a})
=\Tr(\Fr_x,L_1^{\sigma_1'\circ(\tau\tau_a)}),
\]
and left multiplication by $\tau$ permutes the right-$H$-cosets $\tau_aH,\ 1\le a\le h$. If for each $a$, we choose a number $1\le j_a\le n$ such that $L_{j_a}\simeq L_1^{\sigma_1'\circ\tau_a}$, then we may identify $\widetilde{L}_1$ with a semisimple lisse subsheaf of $\s F_1$, namely $L_{j_1}\oplus\cdots\oplus L_{j_h}$. Define
\[
\widetilde{M}_1:=M_{j_1}\oplus\cdots\oplus M_{j_h},
\]
viewed as a lisse subsheaf of $\s F_2$.
Then $\widetilde{L}_1$ and $\widetilde{M}_1$ are $(E,I)$-compatible. By Lemma \ref{L0-supp}, each one of $L_{j_a}$ and $M_{j_a}$, and hence $\widetilde{L}_1$ and $\widetilde{M}_1$, have the same maximal support, denoted by $V$, as that of $L_1$.
Also, since each $L_1^{\sigma_1'\circ\tau_a}$ has the same weight as $L_1$, we see that $\widetilde{L}_1$, as well as $\widetilde{M}_1$, are punctually pure, and they have the same weight.

By \cite[Theorem 3]{Fuj}, we see that $\IC_X(\widetilde{L}_1)$ and $\IC_X(\widetilde{M}_1)$ are weakly $(E,I)$-compatible. Restricting to $V$, we see that $\widetilde{L}_1$ and $\widetilde{M}_1$ are $(E,I)$-compatible on their maximal support (i.e.\ $\rho_*\widetilde{L}_1$ and $\rho_*\widetilde{M}_1$, where $\rho:U\to V$ is the open immersion, are $(E,I)$-compatible).
Then by the additivity of local traces in exact triangles, the quotient perverse sheaves $P_1/\IC_X(\widetilde{L}_1)$ and $P_2/\IC_X(\widetilde{M}_1)$ are weakly $(E,I)$-compatible, and by induction on the length of $P_i$, we are done.
\end{proof}

\begin{sublemma}\label{L0-middle-ext}
\emph{(i)}. Let $j:U\to X$ be an immersion of $\bb F_q$-schemes, and let $\{P_i\}_I$ be a perverse $(E,I)$-compatible system of semisimple mixed perverse sheaves on $U$. Then $\{j_{!*}P_i\}_I$ is perverse $(E,I)$-compatible.

\emph{(ii)}. Let $\{P_i\}_I$ be a perverse $(E,I)$-compatible system of semisimple mixed perverse sheaves on an $\bb F_q$-scheme $X$. Then $\{P_i\}_I$ is weakly $(E,I)$-compatible.
\end{sublemma}

\begin{proof}
(i). Let $U_\alpha\hookrightarrow U$ be a finite collection of essentially smooth irreducible locally closed subschemes in $U$, $\{L_\alpha^i\}_I$ be an $(E,I)$-compatible system of semisimple local systems on $U_\alpha$, such that each irreducible factor of any $L_\alpha^i$ has maximal support $U_\alpha$ in $Cl_U(U_\alpha)$, the closure of $U_\alpha$ in $U$, and that
\[
P_i\simeq\bigoplus_\alpha\ \IC_{Cl_U(U_\alpha)}(L_\alpha^i),
\]
as in Definition \ref{D1} (iii). Then by the transitivity of intermediate extensions \cite[(2.1.7.1)]{BBD}, we have
\[
j_{!*}P_i\simeq\bigoplus_\alpha\ \IC_{Cl_X(U_\alpha)}(L_\alpha^i).
\]

As we saw in the proof of Proposition \ref{L1}, by enlarging the index set for $\alpha$, we may assume that, for each $\alpha$ and $i\in I$, the irreducible factors of $L_\alpha^i$ are Galois conjugate. Also, for each $\alpha$, the lisse sheaves $L_\alpha^i\ (i\in I)$ have the same number of irreducible factors, and these factors are in $(E',I')$-compatible correspondence, for some extension $(E',I')$. Then by Lemma \ref{L0-supp}, for each $\alpha$, the irreducible factors of any one of the $L_\alpha^i$'s all have the same maximal support inside $Cl_X(U_\alpha)$, denoted $V_\alpha$. Finally, it remains to show that, for each $\alpha$, the system $\{L_\alpha^i\}_I$ is $(E,I)$-compatible on $V_\alpha$. Again as above, this follows from \cite[Theorem 3]{Fuj}, since we have reduced to the case where the $L_\alpha^i$'s are punctually pure on $U_\alpha$.

(ii). Let
\[
P_i\simeq\bigoplus_{\alpha}\ \IC_{\ol{X}_{\alpha}} (L^i_{\alpha})
\]
be the decomposition in Definition \ref{D1} (iii). We may replace $P_i$ by $\IC_{\ol{X}_{\alpha}} (L^i_{\alpha})$, and replace $X$ by $\ol{X}_\alpha$.
Changing notations, we now have, for $i=1,2$, $P_i=\IC_X(\s F_i)$, where $\s F_i$ is a semisimple mixed local system on some regular irreducible dense open set $U\subset X$, and $\s F_1$ and $\s F_2$ are $(E,I)$-compatible. So we are back to the situation in the proof of Proposition \ref{L1}.
From the proof above, we have decompositions
\[
\s F_1=\widetilde{L}_1\oplus\cdots\oplus\widetilde{L}_r,\quad \s F_2=\widetilde{M}_1\oplus\cdots\oplus\widetilde{M}_r
\]
into punctually pure sub-local systems, such that $\widetilde{L}_j$ and $\widetilde{M}_j$ are $(E,I)$-compatible with each other. So by \cite[Theorem 3]{Fuj}, we see that $\IC_X(\widetilde{L}_j)$ and $\IC_X(\widetilde{M}_j)$ are weakly $(E,I)$-compatible, thus $P_1$ and $P_2$ are weakly $(E,I)$-compatible.
\end{proof}

\begin{prop}\label{L2}
If $\{K_i\}_I$, for $K_i\in W_m^b(X,\ol{\bb Q}_{\ell_i})$, is a weakly $(E,I)$-compatible system of mixed complexes, then it is weakly perverse $(E,I)$-compatible.
\end{prop}

\begin{proof}
We may assume that $I=\{1,2\}$. Let
\[
[K_i]=[P^+_i]-[P^-_i]
\]
be the canonical representatives. Then each $P_i^\pm$ is a semisimple mixed perverse sheaf \cite[Stabilit\'es 5.1.7 (i)]{BBD}. By Proposition \ref{L1}, it suffices to show that $P_1^+$ and $P_2^+$ are weakly $(E,I)$-compatible: combined with the weak $(E,I)$-compatibility of $K_1$ and $K_2$, this then deduces the weak $(E,I)$-compatibility of $P_1^-$ and $P_2^-$, and one applies Proposition \ref{L1} again.

Again we may assume that $X$ is the union of the closed supports occurring in any one of $P_i^+\oplus P_i^-$, and that $X$ is integral. Then both $P^+_1\oplus P^-_1$ and $P^+_2\oplus P^-_2$ have an open support. Let $U=X^0-Y$, where $X^0$ is the intersection of all open supports occurring in any one of the $P^{\pm}_i$, and $Y$ is the union of all proper closed supports occurring in any one of the $P^\pm_i$. Let
\[
\s F_i^{\pm}:=\s H^{-\dim X}(P_i^{\pm}|_U)=(P_i^\pm|_U)[-\dim X],
\]
which are semisimple mixed lisse sheaves on $U$. Then $\{[\s F^+_i]-[\s F^-_i]\}_I$ is weakly $(E,I)$-compatible on $U$, namely for any $x\in U(\bb F_{q^v})$, we have an equality in $E$
\[
\sigma_1^{-1}\big(\text{Tr(Frob}_x,\s F_1^+)-\text{Tr(Frob}_x,\s
F_1^-)\big)=\sigma_2^{-1}\big(\text{Tr(Frob}_x,\s F_2^+)-\text
{Tr(Frob}_x,\s F_2^-)\big).
\]

By Corollary \ref{C-gal-twist}, applied to the irreducible factors of $\s F_i^\pm\ (i=1,2)$ altogether, there is an extension $(E',I')$ of $(E,I)$, with $E'/E$ Galois, such that the local Frobenius traces of any irreducible factor of $\s F_i^\pm$ are contained in $\sigma_i'(E')$, and that for any embedding $\sigma':E'\to\ol{\Q}_{\ell'},\ \ell'\ne p$, any one of these factors has a $\sigma'$-companion. Then from the equation above, we see that $\s F_1^+\oplus(\s F_2^-)^{\sigma_1'}$ and $\s F_2^+\oplus(\s F_1^-)^{\sigma_2'}$ are $(E',I')$-compatible, and using an $L$-function argument as in the proof of Proposition \ref{L1}, their irreducible factors are in $(E',I')$-compatible one-to-one correspondence.

Since $\s F^+_i$ and $\s F^-_i$ have no irreducible factors in common, one deduces, from the uniqueness of companions, that the irreducible factors of $\s F_1^+$ cannot have their $\sigma_2'$-companions occurring in $(\s F_1^-)^{\sigma_2'}$, and similarly for $\s F_2^+$. Therefore, the irreducible factors of
$\s F_1^+$ and $\s F_2^+$ are in $(E',I')$-compatible one-to-one correspondence, and similarly for $\s F_1^-$ and $\s F_2^-$.

We claim that $\s F_1^+$ has local Frobenius traces contained in $\sigma_1(E)$. As a consequence, $\s F_1^+$ and $\s F_2^+$, and hence $\s F_1^-$ and $\s F_2^-$, are $(E,I)$-compatible. Let $\tau\in\Gal(E'/E)$, also viewed as an automorphism of $\sigma_1'(E')/\sigma_1(E)$. We have, for any $x\in U(\bb F_{q^v})$,
\begin{equation*}
\begin{split}
\Tr(\Fr_x,\s F_1^+)-\Tr(\Fr_x,\s F_1^-)&=
\tau\Tr(\Fr_x,\s F_1^+)-\tau\Tr(\Fr_x,\s F_1^-) \\
&=\Tr(\Fr_x,(\s F_1^+)^{\sigma_1'\circ\tau})-\Tr(\Fr_x,(\s F_1^-)^{\sigma_1'\circ\tau}),
\end{split}
\end{equation*}
thus
\[
\Tr(\Fr_x,\s F_1^+\oplus(\s F_1^-)^{\sigma_1'\circ\tau})
=\Tr(\Fr_x,\s F_1^-\oplus(\s F_1^+)^{\sigma_1'\circ\tau}),
\]
and by Lemma \ref{L-Chebotarev}, $\s F_1^+\oplus(\s F_1^-)^{\sigma_1'\circ\tau}\simeq\s F_1^-\oplus(\s F_1^+)^{\sigma_1'\circ\tau}$. Therefore, since $\s F_1^+$ and $\s F_1^-$ have no irreducible factors in common, we have $\s F_1^+\simeq(\s F_1^+)^{\sigma_1'\circ\tau}$, hence
\[
\Tr(\Fr_x,\s F_1^+)=\Tr(\Fr_x,(\s F_1^+)^{\sigma_1'\circ\tau})
=\tau\Tr(\Fr_x,\s F_1^+).
\]
This finishes the proof of the claim.

Next, we repeat the construction in the proof of Proposition \ref{L1}, with $\s F_i$ therein replaced by $\s F_i^+$ here, to get $(E,I)$-compatible lisse subsheaves $\widetilde{L}_1\subset\s F_1^+$ and $\widetilde{M}_1\subset\s F_2^+$, punctually pure of the same weight. By \cite[Theorem 3]{Fuj}, $\IC_X(\widetilde{L}_1)$ and $\IC_X(\widetilde{M}_1)$ are weakly $(E,I)$-compatible.
This finishes the proof of the proposition: by the additivity of local traces in the Grothendieck group of perverse sheaves, one deduces that
$[P_1^+/\IC_X(\widetilde{L}_1)]-[P_1^-]$
and $[P_2^+/\IC_X(\widetilde{M}_1)]-[P_2^-]$
are weakly $(E,I)$-compatible, and by induction on the length of $P_i^+$, we are done.
\end{proof}

\section{The main result}

\begin{defn}\label{D3-pss}
We say that a complex $\s F\in W^b(X,\Qlb)$ is \textit{perverse
semisimple} if it is isomorphic to a direct sum of shifted
semisimple perverse sheaves, i.e.\
$\s F\simeq\bigoplus_n\leftexp{p}{\s
H^n}(\s F)[-n]$ and each $\leftexp{p}{\s H^n}(\s
F)$ is a semisimple perverse sheaf on $X.$
\end{defn}

\begin{thm}\label{T1}
Let $f:X\to Y$ be a proper morphism of $\bb F_q$-schemes,
and let $\{\s F_i\}_I$ be a strongly perverse $(E,I)$-compatible system of perverse semisimple mixed complexes on $X$. Then $\{Rf_*\s F_i\}_I$ is strongly perverse $(E,I)$-compatible.
\end{thm}

The properness assumption is made, in order to apply Deligne's theorem on weights, as we will see in the following proof.

\begin{proof}
We may assume $I=\{1,2\}$.
By passing to $\leftexp{p}{\s H^n}$ and then to their direct factors as in Definition \ref{D1} (iii), we may assume that $\s F_i=\IC_{\ol
{U}}(L_i)$, where $L_1$ and $L_2$ are $(E,I)$-compatible semisimple local systems on a connected smooth locally closed subset $U\subset X$, whose irreducible factors all have $U$ as their maximal support inside $\ol{U}$. Replacing $X$ by the closure of $U$, we may assume that $U$ is open dense.
Repeating the construction in the proof of Proposition \ref{L1}, with $\s F_1$ and $\s F_2$ therein replaced by $L_1$ and $L_2$ here, we get $(E,I)$-compatible lisse direct summands  $\widetilde{L}\subset L_1$ and $\widetilde{M}\subset L_2$, punctually pure of the same weight. Replacing $L_1$ and $L_2$ by these two direct summands and using induction, we may assume that $L_1$ and $L_2$ themselves are punctually pure, say of weight $w$.

By the purity theorems of Gabber \cite[Corollaire 5.3.2]{BBD} on middle extensions and of Deligne \cite[Proposition 6.2.6]{Del2} on proper pushforwards, $K_i:=Rf_*\IC_X(L_i)$ is a pure complex on $Y$, of weight $w+\dim U$. Therefore by \cite[Corollaire 5.4.4]{BBD}, $\leftexp{p}{\s H^n}(K_i)$ is pure of weight $w+\dim U+n$, for each $n\in\Z$.

By \cite[Theorem 3]{Fuj}, $\IC_X(L_1)$ and $\IC_X(L_2)$ are weakly $(E,I)$-compatible. By the Grothendieck-Lefschetz trace formula (see \cite[Theorem 2]{Fuj}), $K_1$ and $K_2$ are weakly $(E,I)$-compatible, hence are weakly perverse $(E,I)$-compatible, by Proposition \ref{L2}. Let
\[
[K_i]=[P_i^+]-[P_i^-]
\]
be the canonical representatives. From the proof of Proposition \ref{L1}, we have
\[
P_i^+\oplus P_i^-\simeq\bigoplus_\alpha\ \IC_{\ol
{Y}_{\alpha}}(M_\alpha^i),
\]
for some smooth connected locally closed subschemes $Y_\alpha\subset Y$, and for each $\alpha$ an $(E,I)$-compatible system $\{M_\alpha^i\}_I$ of semisimple punctually pure lisse sheaves on $Y_\alpha$, all of whose irreducible factors have maximal support $Y_\alpha$.

Since the perverse cohomology sheaves $\leftexp{p}{\s H^n}(K_i),\ n\in\Z$, have different weights, there is no cancellation in
\[
[K_i]=\sum_n(-1)^n\ [\leftexp
{p}{\s H^n}(K_i)],
\]
namely,
\[
P_i^+\simeq\bigoplus_{n\text{ even}}\leftexp
{p}{\s H^n}(K_i)^{\text{ss}}\qquad \text{and} \qquad
P_i^-\simeq\bigoplus_{n\text{ odd}}\leftexp{p}{\s H^n}(K_i)^{\text{ss}}.
\]
Thus for each $n$, we have
\[
\leftexp{p}{\s H^n}(K_i)^{\text{ss}}\simeq\bigoplus_{\alpha\in A}\ \IC_{\ol{Y}_{\alpha}}(M^i_{\alpha}),
\]
where $A$ is the set of all $\alpha$'s
such that $M_{\alpha}^i$ is punctually pure of weight $w+\dim U+n-\dim Y_{\alpha}$, a number that is independent of $i$. By Lemma \ref{L-poids}, $M_{\alpha}^1$ and $M_{\alpha}^2$ have the same weight, so the set $A$ is independent of $i$. Therefore, $\{\leftexp{p}{\s H^n}(K_i)\}_I$ is perverse $(E,I)$-compatible for each $n$, in other words, $\{K_i\}_I$ is strongly perverse $(E,I)$-compatible.
\end{proof}

The following is a ``geometric" statement, in the sense that it is over the algebraic closure $\bb F$.

\begin{cor}\label{C1}
Let $f:X\to Y$ be a proper morphism of schemes over $\bb F$. Then the supports occurring in the decomposition of $Rf_*\IC_X(\Qlb)$ into shifted irreducible perverse sheaves, as well as the connected monodromy groups of the irreducible local systems on each support, are independent of $\ell\ne p$.
\end{cor}

\begin{proof}
Let $E=\Q$, and $I_0=\{(\ell,\sigma_\ell)|\ell\ne p\}$, where $\sigma_\ell$ is the unique embedding $\Q\to\ol{\Q}_\ell$.
Let $f_0:X_0\to Y_0$ be a model of $f$ over a finite subfield $\bb F_q\subset\bb F$. Then by \cite[Theorem 3]{Fuj}, $\{\IC_{X_0}(\Qlb)\}_{I_0}$ is a weakly $(\Q,I_0)$-compatible system of semisimple pure perverse sheaves on $X_0$, hence perverse $(\Q,I_0)$-compatible by Proposition \ref{L1}.
By Theorem \ref{T1}, $\{Rf_{0*}\IC_{X_0}(\Qlb)\}_{I_0}$ is strongly perverse $(\Q,I_0)$-compatible, so for each $n\in\Z$, we have a decomposition
\[
\big(\leftexp{p}{R^nf_{0*}}\IC_{X_0}(\Qlb)\big)^{\text{ss}}
\simeq\bigoplus _{\alpha}\ \IC_{\ol{Y}_{\alpha,0}}(M^{\ell}_{\alpha,0}),
\]
where for each index $\alpha$ occurring in the direct sum, $\{M^{\ell}_{\alpha,0}\}_{I_0}$ is a $(\Q,I_0)$-compatible system of semisimple mixed lisse sheaves on $Y_{\alpha,0}$, smooth and connected, such that each irreducible factor of $M^\ell_{\alpha,0}$ has $Y_{\alpha,0}$ as its maximal support.

Making a finite base extension $\bb F_{q^v}/\bb F_q$ if necessary, we may assume that all supports $Y_{\alpha,0}$ are geometrically connected. For each $\ell$, making a further base extension $\bb F_{q^{v_\ell}}/\bb F_q$, of degree $v_\ell$ a priori depending on $\ell$, the irreducible factors of $M_{\alpha,0}^\ell$ are geometrically irreducible. We claim that these degrees $v_\ell$ can be taken to be independent of $\ell$. In fact, as we saw in the proof of Proposition \ref{L1}, $(E,I)$-compatible semisimple mixed lisse sheaves have the same number of irreducible factors. Let $\ell_0\ne p$ be a fixed prime. Then all the irreducible factors of $M^{\ell}_{\alpha,0}$, for any other prime $\ell\ne p$, must be geometrically irreducible over $\bb F_{q^{v_{\ell_0}}}$, since if any one of them were to split over a finite extension of $\bb F_{q^{v_{\ell_0}}}$, so would the corresponding factor in $M^{\ell_0}_{\alpha,0}$, contradicting to the choice of $v_{\ell_0}$.

Therefore, by making a finite base extension, we may assume that all irreducible factors of $M^{\ell}_{\alpha,0}$, for all $\ell$ and $\alpha$, are geometrically irreducible.
Then the supports occurring in $\leftexp{p}{R^nf_*}\IC_X(\Qlb)$ are the $Y_{\alpha}=Y_{\alpha,0}\otimes\bb F$, independent of $\ell$.

The connected monodromy groups are by definition the identity components of the Zariski closures of the images of the representations
\[
\rho^{\ell}_{\alpha\beta}:\pi_1(Y_{\alpha})\to GL(V^{\ell}_{\alpha\beta})
\]
corresponding to various irreducible factors $M^{\ell}_{\alpha\beta}$ of $M^\ell_\alpha$ on $Y_{\alpha}$. As we saw in the proof of Proposition \ref{L1}, the irreducible local systems $\{M^{\ell}_{\alpha\beta,0}\}_{I_0}$, for each pair $(\alpha,\beta)$, are $(E,I)$-compatible, punctually pure of some weight $w_{\alpha\beta}$ independent of $\ell$, for some extension $(E,I)$ of $(\Q,I_0)$. Then the last assertion follows from \cite[Theorem 1.6]{Chin}, generalized to higher dimensions (see the remark after \textit{loc.\ cit.}).
\end{proof}

\begin{remark}\label{stack}
(i) Weizhe Zheng suggested to me the following variant of Theorem \ref{T1}. Following \cite[Definition 3.2.1]{Sun-Zheng}, we say that a complex $K\in W^b(X,\Qlb)$ is a \textit{split complex} if it is isomorphic to a direct sum of shifted perverse sheaves.

\begin{subthm}\label{T2}
Let $f:X\to Y$ be a proper morphism of $\bb F_q$-schemes, $w$ be an integer, and let $\{\s F_i\}_I$ be a weakly $(E,I)$-compatible system of split complexes on $X$, each of which is pure of weight $w$. Then $\{Rf_*\s F_i\}_I$ is strongly perverse $(E,I)$-compatible.
\end{subthm}

\begin{proof}
By \cite[Theorem 2]{Fuj} and \cite[Corollary 3.2.5]{Sun-Zheng}, the system $\{Rf_*\s F_i\}_I$ is a weakly $(E,I)$-compatible system of split complexes, pure of weight $w$. By \cite[Corollaire 5.4.4]{BBD}, each $\leftexp{p}{R^n}f_*\s F_i$ is pure of weight $n+w$, so by \cite[Proposition 2.7]{Zheng2}, they are weakly $(E,I)$-compatible, hence perverse $(E,I)$-compatible, by Proposition \ref{L1}.
\end{proof}

(ii) Weizhe Zheng has generalized Theorem \ref{Drinfeld} to Artin stacks with affine stabilizers \cite{Zheng}, as well as \cite[Theorem 3]{Fuj} to Artin stacks (see the end of the article \cite{Zheng2} for the claim; as for the proof, combine \cite[Proposition 5.7]{Zheng2} and \cite[Lemma 6.2]{LO3} to reduce to the case of schemes),
and the author has generalized Deligne's purity theorem for proper pushforwards to such stacks \cite[Proposition 3.9 (iii)]{Sun2}.
Therefore, one can generalize Theorem \ref{T1} (as well as Theorem \ref{T2}) to such stacks as well, and the proof is verbatim:

\vskip.4truecm

\textit{Let $f:X\to Y$ be a proper morphism with finite diagonal between $\bb F_q$-Artin stacks with affine stabilizers. Let $\{\s F_i\}_I$ be a strongly perverse $(E,I)$-compatible system
of perverse semisimple mixed complexes on $X$. Then $\{Rf_*\s F_i\}_I$ is strongly perverse $(E,I)$-compatible.}

\end{remark}

\appendix

\section{Erratum}

Correction to the proof of \cite[Theorem 3.9]{Sun1}.

Yifeng Liu and Weizhe Zheng pointed out that Lemma 5.6.1, and hence the ``K\"unneth formula", Theorem 5.6.5, in \cite{LO1},
are incorrect as stated (even for schemes), and informed me that the proof of \cite[Theorem 3.9]{Sun1}, which cited the K\"unneth formula above on page 82, should be fixed.

Nevertheless, the K\"unneth formula is still valid in the situation therein. In fact, we had reduced to the case where, in the notation in \textit{loc.\ cit.}, $f_1:G\to Y$ is a smooth morphism, and in this case, the K\"unneth formula follows from smooth base change and the projection formula.

\end{document}